\documentclass[10pt]{amsart}
\usepackage{amssymb}
\usepackage{amscd}
\usepackage{mathrsfs}
\usepackage[all]{xy}

\def\today{\number\day\space\ifcase\month\or   January\or February\or
   March\or April\or May\or June\or   July\or August\or September\or
   October\or November\or December\fi\   \number\year}

\theoremstyle{definition}
\newtheorem{thm}{Theorem}[section]
\newtheorem{lem}[thm]{Lemma}

\newtheorem{cor}[thm]{Corollary}

\newtheorem{rmk}[thm]{Remark}

\renewcommand{\qed}{\rule{0.4em}{2ex}}

\newcommand{\beq}{\begin{equation}}
\newcommand{\eeq}{\end{equation}}
\newcommand{\beqr}{\begin{eqnarray*}}
\newcommand{\eeqr}{\end{eqnarray*}}
\newcommand{\bal}{\begin{align*}}
\newcommand{\eal}{\end{align*}}
\newcommand{\bei}{\begin{itemize}}
\newcommand{\eei}{\end{itemize}}

\newcommand{\Z}{{\mathbb{Z}}}

\newcommand{\N}{{\mathbb{N}}}

\newcommand{\K}{{\mathbb{K}}}

\newcommand{\tsr}{{\mathrm{tsr}}}

\title[Stable rank for C*-algebras]
{Stable rank for crossed products by actions of finite groups on C*-algebras}
\author{Hiroyuki Osaka}
\date{\today}
\thanks{$^*$Research of the first author partially supported by the JSPS grant for Scientific Research  No.17K05285}
\address{ Department of Mathematical Sciences\\
  Ritsumeikan University\\ Kusatsu, Shiga, 525-8577  Japan}

\email[]{osaka@se.ritsumei.ac.jp}


\begin{document}

\begin{abstract}

Let $G$ be a finite group, $A$ a unital separable finite simple nuclear C*-algebra, and $\alpha$ an action of 
$G$ on $A$.
Assume that $A$ absorbs the Jiang-Su algebra $\mathcal{Z}$,
the extremal boundary of the trace space of $A$ is compact and finite dimensional and that 
$\alpha$ fixes any tracial state of $A$. Then $\tsr(A \rtimes_\alpha G) = 1$.
In particular, when $A$ has a unique tracial state, we conclude it without above conditons on a tracial state space of $A$.

\end{abstract}

\maketitle

\section{Introduction}

Rieffel \cite{Rf83} defined the (topological) stable rank, $\tsr(A)$, of a C*-algebra, which is the noncommutative 
analogue of the complex dimension of topological spaces. That is, for the continuous functions on a compact Hausdorff 
$X$ one has $\tsr(C(X)) = \lfloor\dfrac{1}{2}\dim X\rfloor + 1$, where $\dim X$ is the covering dimension of $X$. 
For a unital C*-algebra $A$ the stable rank $\tsr(A)$ is either $\infty$ or the smallest possible integer $n$ such that each $n$-tuple in $A^n$ 
can be approximated in norm by $n$-tuples $(b_1, \dots, b_n)$ such that $\sum_{i=1}^nb_i^*b_i$ is invertible. 
For a nonunital C*-algebra $A$ we define $\tsr(A) = \tsr(\tilde{A})$, $\tilde{A}$ means the unitaization of $A$.

Rieffel \cite{Rf83} showed that $\tsr(A) = 1$ if and only if $\tsr(M_n(A)) = 1$ for $n \in \N$ and $\tsr(A) = 1$  
if and only if $\tsr(A \otimes \K) = 1$ for the C*-algebra $\K$ of compact operators on a separable infinite dimensional Hilbert space.
Related to crossed products we have in general that $\tsr(A \rtimes_\alpha \Z) \leq \tsr(A) + 1$ by \cite[Theporem~7.1]{Rf83} and 
$\tsr(A \rtimes_\alpha G) \leq \tsr(A) + \mathrm{card}(G) -1$ for any action $\alpha$ from a finite group $G$ on $A$ by \cite[Theorem~2.4]{JOPT}. However, it is not easy to determine when those crossed products  have stable rank one. except crossed products by "strongly" outer actions as (tracial) Rokhlin property \cite{LS15}, \cite{OT17}, \cite{NCP06}.
See \cite{O03}, \cite{AK},  \cite{LB} and their bibliography for basic properties about stable ranks. 

In this note we determine stable rank one for crossed products $A \rtimes_\alpha G$ of any action $\alpha$ from a finite group $G$ on a separable finite  simple nuclear unital C*-algebra $A$,
when $A$ is the Jiang-Su absorbing with some conditions on a tracial state space of $A$, using observations by Sato \cite{sato16} and  by R\o rdam \cite{Rordam:Z-absorbing}. In particular, when $A$ has a unique tracial state, we conclude it without above conditions on a tracial state space of $A$.
Here. the Jiang-Su algebra $\mathcal{Z}$ is a unital separable simple infinite-dimensional nuclear C*-algebra with a unique tracial state whose K-theoretic invariants are same as that of the complex numbers \cite{JS99}.
In the current classification theorem of C*-algebras, the absorption of $\mathcal{Z}$ is regarded as one of the regular properties of classiable C*-algebras. See \cite{EGLN}, \cite{GLN}, \cite{TWW}.

\section{Stable rank for inclusions of unital C*-algebras}

Let $A \subset B$ be an inclusion of unital C*-algebras and 
$E \colon G \rightarrow A$ be a conditional expectation of 
index-finite type  as defined in Definition 1.2.2. of \cite{Watatani:index}.  

The following is a general formula for stable rank for an inclusion of unital C*-algebras of  index-finite type.

\vskip 2mm

\begin{thm}\label{thm:rank}\cite[Theorem~2.1]{JOPT}
$1 \in A \subset B$ of unital C*-algebras,
let $E \colon B \to A$ be
a conditional expectation with index-finite type,
and let $\big( (v_k, v_k^*) \big)_{1 \leq k \leq n}$
be a quasi-basis for $E.$
Then $\tsr (B) \leq \tsr (A) + n - 1.$
\end{thm}

\vskip 2mm

The inclusion $1 \in A \subset B$ of unital C*-algebras of index-finite type is said to have finite depth $k$ 
if the derived tower obtained by iterating the basic construction
$$
A' \cap A \subset A' \cap B \subset A' \cap B_2 \subset A' \cap B_3 \subset \cdots
$$
satisfies $(A' \cap B_k)e_k(A' \cap B_k) = A ' \cap B_{k+1}$, 
where $\{e_k\}_{k \in \N}$ are projections obtained by iterating the basic construction, so that $B_1 = B$, $e_1 = e_A$, 
and $B_{k+1} = C^*(B_k, e_k)$. When $G$ is a finite group and $\alpha$ an action of $G$ on $A$, 
it is well known that an inclusion $1 \in A \subset A \rtimes_\alpha G$ is of depth $2$. (See \cite[Lemma 3.1]{OT06}.) 

In the case of an infinite dimensional  simple unital C*-algebra $A$ with Property (SP), that is , any nonzero heriditary C*-subalgebra of $A$ has nonzero projection,  we have the following estimate.

\vskip 2mm

\begin{thm}\label{thm:simple}\cite[Theorem~3.2]{OT07}
$1 \in A \subset B$ of unital C*-algebras of index-finite type and depth $2$. 
Suppose that $A$ is an infinite dimensional simple C*-algebra with $\tsr(A) = 1$ and Property (SP).
Then $\tsr(B) \leq 2$.
\end{thm}

\vskip 2mm

\begin{rmk}
When $A$ is not simple, the estimate in Theorem~\ref{thm:rank} is the best possible. 
Indeed, in \cite[Example~8.2.1]{B:symmetry} Blackadar constructed a symmetry action $\alpha$ on the CAR $\mathcal{U}$ algebra such that 
$$
(C[0, 1] \otimes \mathcal{U}) \rtimes_{id \otimes \alpha}\Z/2\Z \cong C[0, 1] \otimes B,
$$
where $\tsr(B) = 1$ and $K_1(B) \not= 0$. Hence we know that $\tsr(C[0, 1] \otimes B) = 2$ 
by \cite[Corollary~7.2]{Rf83} and \cite[Proposition~5.2]{NOP}.
\end{rmk}

\vskip 3mm

Let $\mathcal{Z}$ be the Jiang-Su algebra. 
When a C*-algebra $B$ in Theorem~\ref{thm:simple} is Jiang-Su absorption, that is,  $A \otimes \mathcal{Z} \cong A$, we can conclude that $\tsr(B) =1$ as follows.

\vskip 2mm

\begin{thm}\label{thm:inclusion}
$1 \in A \subset B$ of unital C*-algebras of index-finite type. 
Suppose that $A$ is an infinite dimensional  simple C*-algebra with $\tsr(A) < \infty$ 
and $B$ is Jiang-Su absorption.
Then $\tsr(B) = 1$.
\end{thm}

\vskip 2mm

We use the following simple observation  to prove Theorem~\ref{thm:inclusion}

\vskip 2mm 

\begin{lem}\label{lem:infinite}
Let $A$ be a simple unital C*-algebra. 
Then $A$ is finite if $\tsr(A) < \infty$.
\end{lem}

\begin{proof}
Suppose that $A$ is infinite. Then from \cite{Cu77} there are orthogonal projections $p, q$ such that 
$1 \sim p \sim q$, where $\sim$ means the Murray-von Neumann equivalence.
Hence $tsr(A) = \infty$ by \cite[Proposition~6.5]{Rf83}, and a contradiction.
\end{proof}

{\bf Proof of Theorem~2.4}

Since $A$ is simple, $B$ can be decomposed into finite direct sums $\oplus B_i$ of simple (unital) closed ideals by
\cite[Theorem~3.3]{Izumi:inclusion}.  By Theorem~\ref{thm:rank} we know $\tsr(B) < \infty$, that is, 
$\tsr(B_i) < \infty$ for each $i$. Hence 
each $B_i$ is a finite simple C*-algebra by Lemma~\ref{lem:infinite}. 

Let $\mathcal{Z}$ be the Jiang-Su algebra. Then, since each $B_i$ is a finite simple C*-algebra,  by \cite[Theorem~6.7]{Rordam:Z-absorbing}
each $B_i \otimes \mathcal{Z}$ has stable rank one. 
From the assumption since 
\begin{align*}
B &\cong B \otimes \mathcal{Z}\\
&= \oplus_i B_i \otimes \mathcal{Z}, 
\end{align*}
we conclude that $\tsr(B) = \max_i\{\tsr(B_i \otimes \mathcal{Z})\} = 1$.
\hfill$\qed$

\vskip 2mm

\section{Stable rank for C*-crossed products}

Very recently, Sato \cite{sato16} gives the sufficient condition for the Jiang-Su absorption of crossed products by actions of amenable groups on $\mathcal{Z}$-absorbing C*-algebras.
Using this observation we can prove the stable rank one property for the crossed product $A \rtimes_\alpha G$  by an action $\alpha$ of  a finite group on a $\mathcal{Z}$-absobing C*-algebra $A$ under some condition.



\vskip 2mm

\begin{thm}\label{thm:main theorem}
Let $G$ be a finite group, $A$ a unital separable finite simple nuclear C*-algebra, and $\alpha$ an action of 
$G$ on $A$.
Assume that $A$ absorbs the Jiang-Su algebra $\mathcal{Z}$,
the extremal boundary of the trace space of $A$ is compact and finite dimensional and that 
$\alpha$ fixes any tracial state of $A$. Then $\tsr(A \rtimes_\alpha G) = 1$.
\end{thm}

\vskip 2mm

\begin{proof}
Note that an inclusion $A \subset A \rtimes_\alpha G$ is of a finite-index type. 

By \cite[Theorem~1.1]{sato16} $A \rtimes_\alpha G$ is the Jiang-Su absorbing.  
Since $A$ is a finite simple unital C*-algebra with $A \otimes \mathcal{Z} \cong A$, $\tsr(A) =1$ by
\cite[Theorem~6.7]{Rordam:Z-absorbing}. Hence, by Theorem~\ref{thm:inclusion} 
$\tsr(A \rtimes_\alpha G) = 1$.
\end{proof}

\vskip 2mm

\begin{rmk}
When $A$ is a unital simple C*-algebra with $\tsr(A) = 1$ and Property (SP), then we know that 
$\tsr(A \rtimes_\alpha G) \leq 2$ for any action $\alpha$ of a finite group on $A$ by Theorem~\ref{thm:simple}.
Moreover, if $\alpha$ is "strongly" outer like the tracial Rokhlin property in the sense of N.C.Phillips \cite{NCP06}, then $\tsr(A \rtimes_\alpha G) = 1$.
(For example see \cite[Proposition~4.13]{OT17}.)
\end{rmk}

\vskip 2mm

\begin{cor}
Let $G$ be a finite group, $A$ a unital separable finite simple nuclear $\mathcal{Z}$-absorbing C*-algebra with a unique tracial state, and $\alpha$ an action of $G$ on $A$. Then $\tsr(A \rtimes_\alpha G) = 1$.
\end{cor}

\vskip 2mm






{\it Acknowledgments}

The author would like to thank Yasuhiko Sato for helpful conversations.


\end{document}